\documentclass[12pt,leqno,fleqn]{amsart}
\usepackage{amsmath,amstext,amsthm,amssymb,amsxtra}
\usepackage{txfonts,pxfonts} 
\usepackage[colorlinks,citecolor=red,pagebackref,hypertexnames=false]{hyperref}
\usepackage[backrefs]{amsrefs}
\allowdisplaybreaks
\swapnumbers

\theoremstyle{plain} 
\newtheorem{lemma}[equation]{Lemma}
\newtheorem{proposition}[equation]{Proposition}
\newtheorem{theorem}[equation]{Theorem}

\theoremstyle{definition}
\newtheorem{definition}[equation]{Definition}

\theoremstyle{remark}

\newtheorem{remark}[equation]{Remark}

\numberwithin{equation}{section}

\def\norm#1.#2.{\lVert#1\rVert_{#2}}
\def\Norm#1.#2.{\bigl\lVert#1\bigr\rVert_{#2}}
\def\NOrm#1.#2.{\Bigl\lVert#1\Bigr\rVert_{#2}}
\def\NORm#1.#2.{\biggl\lVert#1\biggr\rVert_{#2}}
\def\NORM#1.#2.{\Biggl\lVert#1\Biggr\rVert_{#2}}

\def\ip#1,#2,{\langle #1,#2\rangle}
\def\Ip#1,#2,{\bigl\langle#1,#2\bigr\rangle}
\def\IP#1,#2,{\Bigl\langle#1,#2\Bigr\rangle}

\def\Abs#1{\bigl\lvert#1\bigr\rvert}
\def\ABs#1{\biggl\lvert#1\biggr\rvert}

\def\XXint#1#2#3{{\setbox0=\hbox{$#1{#2#3}{\int}$}
     \vcenter{\hbox{$#2#3$}}\kern-.5\wd0}}

\def\diag{\textnormal{diag}}

\def\d{\textnormal {d}}

\begin{document}
\title[The maximal function along a polynomial curve]{Logarithmic Dimension bounds for the maximal function along a polynomial curve}

\thanks{Research partially supported by the Wallenberg Foundation and by the Fields Institute.}

\author[Ioannis Parissis]{Ioannis Parissis}
\email{ioannis.parissis@gmail.com}
\address{Institutionen f\"or Matematik,
Kungliga Tekniska H\"ogskolan,
SE 100 44, Stockholm, Sweden.}

\begin{abstract}
Let $\mathcal M$ denote the maximal function along the polynomial curve $(\gamma_1t,\ldots,\gamma_dt^d)$:
 $$\mathcal M(f)(x)=\sup_{r>0}\frac{1}{2r}\int_{|t|\leq r}|f(x_1-\gamma_1 t,\ldots,x_d-\gamma_d t^d)|dt.$$
We show that the $L^2$ norm of this operator grows at most logarithmically with the parameter $d$:
$$\norm \mathcal M f.L^2(\mathbb R^d). \leq c \log d \ \norm f .L^2(\mathbb R^d).,$$
where $c>0$ is an absolute constant. The proof depends on the explicit construction of a ``parabolic'' semi-group of operators which is a mixture of stable semi-groups.
\end{abstract}
\maketitle
\thispagestyle{empty}
\section{Introduction}
For $\Gamma=\diag(\gamma_1,\gamma_2,\ldots,\gamma_d)$ with $\gamma_1,\gamma_2,\ldots,\gamma_d \in \mathbb{R}\setminus\{0\} $, we define the measures $d\mu_r ^\Gamma$ as
\begin{eqnarray*}
\ip \phi,d\mu_r ^\Gamma, = \frac{1}{2r}\int _{|t|\leq r} {\phi(\gamma_1t,\gamma_2t^2,\ldots,\gamma_d t^d)}\ dt, \ \ \phi\in\mathcal{S}(\mathbb R ^d).
\end{eqnarray*}
For $f\in \mathcal S (\mathbb R ^d)$ the maximal function along the polynomial curve $(\gamma_1t,\ldots,\gamma_dt^d)$ is defined as
\begin{eqnarray*}
\mathcal M_{\mu^\Gamma} (f)(x)&=&\sup_{r>0}\frac{1}{2r}\int_{|t|\leq r}|f(x_1-\gamma_1t,x_2-\gamma_2t^2,\ldots,x_d-\gamma_dt^d)|\ dt\\&=&\sup_{r>0}(|f|*d\mu_r ^\Gamma)(x).
\end{eqnarray*}
Our main result gives a control on the norm of this operator in $L^2(\mathbb R ^d)$ in terms of the parameter $d$. In particular we have:
\begin{theorem}\label{t.main} There exists an absolute constant $c>0$ such that for every $f\in L^2(\mathbb R^d)$ we have
\begin{eqnarray*}
\norm \mathcal M _{\mu^\Gamma} f. L^2(\mathbb R^d). \leq c \log d \ \norm f. L^2(\mathbb R^d). .
\end{eqnarray*}
\end{theorem}

The method we use to prove Theorem \ref{t.main} is mainly inspired by Bourgain's work on the dimension free bounds for the maximal function associated with a convex body. In \cite{B} for example, Bourgain compares the characteristic function of the convex body with the Poisson semi-group, the latter being controlled by Stein's general maximal theorem for symmetric diffusion semi-groups. This approach also appears for example in \cite{SS} where the authors use the Heat semi-group instead, and in \cite{B1} and \cite{Car}. However, these semi-groups, adopted to the Euclidean isotropic structure, are not compatible with the parabolic dilations we are considering. The appropriate Poisson kernel for the space of homogeneous type under study is a mixture of stable semi-groups. In an abstract setting (homogeneous groups, symmetric spaces) the existence of such semi-groups is well known. See for example \cite{S1}, \cite{S2}     and \cite{SF}. In this paper we construct such a Poisson kernel explicitly, our starting point being essentially the desired properties of its Fourier transform.

The paper is organized as follows. In Section \ref{s.dil} we introduce the notion of parabolic dilations and we define the associated norm that will accompany us throughout the paper. The definition of this norm might seem a bit unmotivated at that point but its usefulness will become clear later on. In Section \ref{s.parmax} we look into more detail at our main object of study, the maximal operator along a polynomial curve. We explain how the problem reduces to studying the corresponding square function and a ``parabolic`` semigroup of operators compatible with the parabolic dilations. This semigroup is also discussed in this section along with its main properties. The proof of the main theorem is concluded in Section \ref{s.square} where the necessary oscillatory integral estimates are also stated and proven.

\section{Notations} Throughout the paper $c$ will denote a numerical positive constant which might change even in the same line of text. We will many times suppress numerical constants by using the symbol $\lesssim$. Thus $A\lesssim B$ means that $A\leq c B$ for $c$ as described. We will never suppress constants that depend on $d$.

Since we are dealing with positive operators, we will always assume that the symbol $f$ stands for a non negative function. We will use this assumption without any further comment in what follows.

Finally, for every Lebesgue measurable set $K\subset \mathbb R^d$ we will write $|K|$ for its Lebesgue measure.
\section{Parabolic Dilations} \label{s.dil}

We will work on the Euclidean space $\mathbb R^d$ endowed with the family of dilations
\begin{eqnarray}\label{e.dilations}
\delta_s x  =(sx_1,s^2x_2,\ldots,s^d x_d), \ \ x\in\mathbb{R}^d, s>0.
\end{eqnarray}
We will call $\delta_s$ the \emph{parabolic dilations operator}. We will now define a norm function that is homogeneous with respect to the dilations \eqref{e.dilations} in the following way. We fix a positive integer $n$ such that $2^{n-1}<d \leq 2^n$ and write
\begin{eqnarray}\label{e.metric}
\rho(x)=\sum_{0\leq l\leq n-1} \bigg(\sum_{2^{l-1}<j\leq 2^l} |x_j|^\frac{2^l}{j}\bigg)^\frac{1}{2^l}+\bigg(\sum_{2^{n-1}<j\leq d} |x_j|^\frac{2^n}{j}\bigg)^\frac{1}{2^n}.
\end{eqnarray}
The following Proposition contains the basic properties of the norm function $\rho$:

\begin{proposition}\label{p.metric} Let $\rho:\mathbb R^d\rightarrow [0,\infty)$ be the function defined in \eqref{e.metric}. Then $\rho$ satisfies the following properties for every $x,y\in \mathbb R ^d$:
\begin{enumerate}
\item[(i)]{$\rho(x)=0\Leftrightarrow x=0$}.
\item[(ii)]{For $s>0$, $\rho(\delta_s x)=s \rho(x)$.}
\item[(iii)]{$\rho(-x)=\rho(x)$.}
\item[(iv)]{$\rho(x+y)\lesssim \rho(x)+\rho(y)$.}
\item[(v)]{The function $\rho$ is continuous in $\mathbb R^d$.}
\end{enumerate}
For $x,y\in\mathbb R^d$ we can thus define $\rho^*(x,y)\coloneqq \rho(x-y)$ and $\rho^*$ is a translation invariant metric in $\mathbb R ^d$.
\end{proposition}

Set $\alpha=1+2+\cdots+\d$. We usually refer to $\alpha$ as the \emph{homogeneous dimension} of $\mathbb R^ d$ since $\mathbb{R}^d$ endowed with a parabolic norm $\rho$ and the usual Lebesgue measure is a space of homogeneous type in the sense of Coifman and Weiss \cite{CW}. Since the ``balls'' $\{x\in \mathbb{R}^d: \rho(x)< r\}$ have measure of the order $r^\alpha$, $\alpha$ is the homogeneous dimension of the space.

We introduce now a system of polar coordinates which is appropriate in the context of parabolic dilations. For $x\in\mathbb{R}^d\setminus\{0\}$ we consider
\begin{eqnarray*}
x^\prime = \delta _{\rho(x)^{-1}} x \in S^\rho,
\end{eqnarray*}
where $S^\rho$ is the ``unit sphere'' corresponding to the parabolic norm $\rho$:
\begin{eqnarray*}
S^\rho=\{x\in\mathbb R ^d: \rho(x)=1\}.
\end{eqnarray*}

The polar coordinates are then defined by the mapping
\begin{eqnarray*}
\mathbb R ^d\setminus \{0\}\ni x \mapsto (x^\prime, \rho(x))\in S^\rho \times (0,\infty).
\end{eqnarray*}
We have the following lemma which is classical in this area.
\begin{lemma}\label{l.polar} There is a unique Radon measure $\sigma_\rho$ on $S^\rho$ such that for all $\phi\in L^1(\mathbb R ^d)$
\begin{eqnarray*}
\int_{\mathbb R^d}\phi(x)dx=\int_0 ^\infty \int_{S^\rho} \phi(\delta_rx^\prime)r^{\alpha-1} d\sigma_\rho(x^\prime) dr.
\end{eqnarray*}
\end{lemma}
For the proof of this lemma see for example \cite{SF}.

We introduce the (parabolic) dilations of a function $K\in L^1(\mathbb{R}^d)$. We will consistently use the notation $K_s$ to denote the anisotropic dilations of the function $K$:
\begin{eqnarray*}
K_s(x)=\frac{1}{s^\alpha}K(\delta_s ^{-1}x), \ \ s>0.
\end{eqnarray*}
Of course we have $\widehat {K_s} (\xi) = \hat K(\delta_s \xi)$ for $s>0$ and $\xi\in  \mathbb{R}^d$. In particular $\int_{\mathbb{R}^d} K_s =\int_{\mathbb{R}^d} K$ for any $s>0$.

Likewise, if $\mu$ is a finite Borel measure we can define the parabolic dilation of $\mu$ to be the measure $\mu_s$ acting on test functions $\phi\in\mathcal S (\mathbb R ^d)$ as:
\begin{eqnarray*}
\ip \phi , \mu_s, =\int \phi(\delta_s x)d\mu(x), \ s>0.
\end{eqnarray*}
Equivalently we can define $\mu_s$ by the relation
\begin{eqnarray*}
\widehat {\mu_s}(\xi) =\hat\mu(\delta_s\xi), \ \ s>0, \ \ \xi\in\mathbb R ^d.
\end{eqnarray*}
We stress that $K_r$ and $d\mu_r$ always denote the \textbf{parabolic} dilations of a function or a measure
respectively (not to be confused with the standard dilations of $\mathbb R^d$ many times written with the same
notation). \label{p.pardil}
\section{The maximal function along a polynomial curve and a related semigroup of operators.}\label{s.parmax}

\subsection{The maximal function along a polynomial curve}
We are interested in the following model-case operator that controls differentiation along a polynomial curve. For  $\Gamma=\diag(\gamma_1,\gamma_2,\ldots,\gamma_d)$ with $\gamma_1,\gamma_2,\ldots,\gamma_d \in \mathbb{R}\setminus\{0\} $, we define the measure $d\mu^\Gamma$ as
\begin{eqnarray*}
\ip \phi,d\mu^\Gamma, = \frac{1}{2}\int _{|t|\leq 1} {\phi(\gamma_1t,\gamma_2t^2,\ldots,\gamma_dt^d)}\ dt, \ \ \phi\in\mathcal{S}(\mathbb R ^d).
\end{eqnarray*}
For $f\in \mathcal S(\mathbb R^d)$ the maximal function along the polynomial curve $(\gamma_1t,\ldots,\gamma_dt^d)$ is defined as
\begin{eqnarray}
\mathcal M_{\mu^\Gamma}(f)(x)&=&\sup_{r>0}(f*d\mu^\Gamma _r)(x)\\&=&\sup_{r>0}\frac{1}{2r}\int_{|t|\leq r}f(x_1-\gamma_1t,x_2-\gamma_2t^2,\ldots,x_d-\gamma_dt^d)\ dt.
\end{eqnarray}

The operator $\mathcal M_{\mu^\Gamma}$ has been well studied. In fact it is well known that this operator is bounded on $L^p(\mathbb R^d)$, $1<p\leq \infty$ and the norm of the operator depends only on the parameter $d$. For further details about the $L^p$ bounds we refer the reader to \cite{SW}. An alternative proof of the $L^p$ bounds is contained in \cite{DR}. A standard reference that contains these results as well as corresponding results for more general curves is the book of Stein \cite{S}. Finally we note that behavior of this operator in $L^1(\mathbb R ^d)$ is not very well understood. The question whether this operator is bounded from $L^1(\mathbb R^d)$ to weak $L^1(\mathbb R^d)$ is open. For results ``close'' to $L^1$ we refer the interested reader to \cite{C}, \cite {CS} and \cite{ST}. The purpose of this note is to study the dependence of the $L^2$ bounds of this operator on the parameter $d$.

We first note some easy reductions. One immediately makes the observation that it is enough (in terms of $L^p$-boundedness) to take $\Gamma=I$ and study the operator
\begin{eqnarray}
\mathcal{M_\mu}(f)(x)=\sup_{r>0}(f*d\mu _r)(x)=\sup_{r>0}\frac{1}{2r}\int_{|t|\leq r}f(x_1-t,x_2-t^2,\ldots,x_d-t^d)\ dt,
\end{eqnarray}
where $\mu=\mu^I$. Indeed we have that
\begin{eqnarray*}
\mathcal{M}_{\mu^\Gamma}(f)=\mathcal{M}_\mu(f\circ \Gamma)\circ \Gamma^{-1},
\end{eqnarray*}
and so $\norm \mathcal {M}_{\mu^\Gamma}.p\rightarrow p. = \norm \mathcal {M}_\mu .p\rightarrow p.$ for any $1<p\leq \infty$.

A further reduction can be made by observing that the operator $\mathcal M_\mu$ is a lacunary maximal operator in disguise. To see this we define the probability measure
\begin{eqnarray}\label{e.sigma}
\hat\sigma(\xi) = \int _{\frac{1}{2}<|t|\leq 1} e^{-2\pi i (\xi_1t+\cdots+\xi_d t^d)}dt,
\end{eqnarray}
and for $k\in \mathbb Z$ we consider the -$2^k$- parabolic dilations of $\sigma$:
\begin{eqnarray*}
\widehat{\sigma_{2^k}}(\xi) = \hat\sigma(\delta_{2^k}\xi)=\frac{1}{2^k}\int _{2^{k-1}<|t|\leq 2^k} e^{-2\pi i (\xi_1t+\cdots+\xi_d t^d)}dt.
\end{eqnarray*}
Now we fix some $r>0$ and write $2^{k_o-1}<r\leq 2^{k_o}$. We have
\begin{eqnarray*}
\frac{1}{2r}\int_{|t|\leq r}f(x_1-t,x_2-t^2,\ldots,x_d-t^d)\ dt&\leq&
\frac{1}{2^{k_o}} \sum_{k=-\infty} ^{k_o} 2^k f*d\sigma_{2^k}(x)\leq 2 \sup_{k\in\mathbb Z} (f*d\sigma_{2^k})(x).
\end{eqnarray*}
Thus, if we define $\mathcal M_ \sigma ^d(f)(x)\coloneqq\sup_{k\in\mathbb Z} (f*d\sigma_{2^k})(x)$, we have
\begin{eqnarray*}
\mathcal M_\sigma ^d(f)(x)\leq \mathcal{M}_\mu(f)(x)\leq 2 \mathcal M_\sigma ^d(f)(x).
\end{eqnarray*}
This means that it is enough to study the dyadic maximal operator $\mathcal M^d _\sigma$ on $L^2(\mathbb R^d)$.

\subsection{A related semigroup of operators}

Let us for a moment consider a general maximal operator of the form
\begin{eqnarray*}
\mathcal M^d _\lambda(f)(x)=\sup_{k\in \mathbb Z} (f*d\lambda_{2^k})(x), \ \ f\in S(\mathbb R ^d),
\end{eqnarray*}
where $d\lambda$ is some probability measure on $\mathbb R ^d$. As usual, $d\lambda_t$ denotes the parabolic dilations of the measure $d\lambda$. We would like to replace the supremum in the maximal function above by an appropriate square function, that is, write down an estimate of the form
\begin{eqnarray}\label{e.toy}
\mathcal M^d _\lambda(f)(x)\leq \big(\sum_{k\in \mathbb Z} (f*d\lambda_{2^k})^2(x)\big)^\frac{1}{2}.
\end{eqnarray}
Suppose that $\hat \lambda$ has some decay at infinity, and this will indeed be the case in the problem we are interested in, so that the above sum makes sense  at infinity. Of course $\hat \lambda(0)=1$ since $d\lambda$ is a probability measure so we can't expect anything like that close to $0$. In order to make the Fourier transform of the measure small close to zero we replace \eqref{e.toy} by the estimate
\begin{eqnarray}\label{e.notoy}
\mathcal M^d _\lambda(f)(x)\leq \sup_{k\in \mathbb Z}(f*d\eta_{2^k})(x)+ \big(\sum_{k\in \mathbb Z} |f*(d\lambda-d\eta)_{2^k}|^2(x)\big)^\frac{1}{2}.
\end{eqnarray}
for some suitable probability measure $d\eta$. So, in order to make this estimate useful $\hat\eta$ has be chosen so that it decays fast enough at infinity. Of course this is a very old trick used to estimate maximal functions. Since we are interested in the operator norms that appear in \ref{e.notoy}, the choice of the measure $d\eta$ must take that into account. Following Bourgain from \cite{B} we will define an appropriate measure $d\eta$ so that the dilations $d\eta_t$ give rise to a symmetric diffusion semigroup of operators. Then by an appeal to Stein's general maximal theorem for symmetric diffusion semi-groups \cite{S1}, the maximal function associated with the measure $d\eta$ is bounded on all $L^p$ spaces with constants that do not depend on the dimension.

The following lemma defines the appropriate ``Poisson Kernel'' for the space of homogeneous type we are considering. Remember that $\rho$ is the parabolic norm defined by \eqref{e.metric}.

\begin{lemma} \label{l.core} There exists a probability measure $dP^\rho$ on $\mathbb R^d$ such that $\widehat {dP^\rho}(\xi)=e^{-\rho(\xi)}$.
\end{lemma}
The proof of this lemma is not difficult but it relies on several classical results from Probability theory. First, we recall the notions of positive definite and negative definite functions.

\begin{definition}[Positive Definite Functions] A function $\psi:\mathbb R^d\rightarrow \mathbb C$ is called positive definite if for any choice of $m\in\mathbb N$ and any choice of vectors $x_1,\ldots,x_m\in \mathbb R^d$ and complex numbers $\lambda_1,\ldots,\lambda_m$
\begin{eqnarray*} \sum_{j=1} ^{m}\sum_{k=1} ^{m} f(x_j-x_k)\lambda_j \bar\lambda_k \geq 0.
\end{eqnarray*}
\end{definition}

The notion of positive definite functions is closely related to the notion of negative definite functions.

\begin{definition}\label{d.ndef}[Negative Definite Functions] Suppose that $\psi:\mathbb R^d\rightarrow \mathbb C$ is a measurable function. Then $\psi$ is called negative definite if it satisfies the following properties
\begin{enumerate}
\item[(i)] $\psi(0)\geq 0$.
\item[(ii)] The function $\psi$ is Hermitian: $\psi(-x)=\overline{\psi(x)}$ for every $x\in\mathbb R^d$.
\item[(iii)] For any $m\in \mathbb N$ and any choice of vectors $x_1,\ldots,x_m\in\mathbb R^d$ and complex numbers $\lambda_1,\ldots,\lambda_m$, $\sum_{j=1} ^m \lambda_j=0$ implies that
\begin{eqnarray*} \sum_{j=1} ^{m}\sum_{k=1} ^{m} f(x_j-x_k)\lambda_j \bar\lambda_k \leq 0.
\end{eqnarray*}
\end{enumerate}
\end{definition}

The relation between positive definite and negative definite functions is the content of \emph{Schoenberg's} theorem:

\begin{theorem}\label{t.schoen} A function $\psi:\mathbb R^d\rightarrow \mathbb C$ is (continuous and) negative definite if and only if $\psi(0)\geq0$ and $e^{-t\psi(x)}$ is (continuous and) positive definite for every $t\geq0$.\footnote{In the literature many times this is the actual definition of negative definite functions and our definition is another characterization. Note also that in the literature our notion of negative definite functions is many times refered to as ``conditionally negative definite functions''.}
\end{theorem}

This theorem is classical. A proof can be found in \cite{J}. In general we refer the reader to \cite{J} for an excellent exposition on the notions of negative definite functions and their relation to convolution semi-groups.

Of course Bochner's classical theorem states that a function $\psi$ is the Fourier transform of a probability measure on $\mathbb R^d$ if and only if it is positive definite on $\mathbb R^d$ and continuous at $0$ with $\psi(0)=1$ . See for example \cite{Sad} or \cite{J} for a proof and details of this classical result.
The following simple lemma allows us to raise negative definite functions which are non negative to ``small'' exponents while remaining in the class of negative definite functions.

\begin{lemma}\label{l.power} Let $\psi:\mathbb R^d\rightarrow [0,\infty)$ be a negative definite function and let $0<\gamma<1$. Then the function $\psi^\gamma$ is also negative definite.
\end{lemma}
\begin{proof} The proof relies on the simple identity involving the $\Gamma-$function:
\begin{eqnarray}\label{e.subord}
x^\gamma=\frac{\gamma}{\Gamma(1-\gamma)} \int_0 ^\infty(1-e^{tx})t^{-\gamma-1}dt,
\end{eqnarray}
for all $x\geq0$ and $\gamma\in(0,1)$. Indeed, since $\psi$ is negative definite, Schoenberg's Theorem implies
that $e^{-t\psi(x)}$ is positive definite for all $t\geq 0$ . Using this fact and Definition \ref{d.ndef} one
readily sees that $1-e^{-t\psi(x)}$ is negative definite for all $t>0$. We conlcude that $\int_0 ^\infty
(1-e^{-t\psi(x)})d\mu(t)$ is also negative definite for all positive measures $d\mu$ on $(0,\infty)$ which are
finite away from $0$. Now since $\psi\geq 0$ we can use \eqref{e.subord} to write
\begin{eqnarray*}
\psi^\gamma=\frac{\gamma}{\Gamma(1-\gamma)}  \int_0 ^\infty(1-e^{-t\psi})t^{-\gamma-1}dt,
\end{eqnarray*}
and this concludes the proof since the measure $t^{-\gamma-1}$ is positive and finite away from the origin.
\end{proof}
\begin{remark} What's behind Lemma \ref{l.power} is the fact that the function $s\mapsto s^\gamma$ is a \emph{Bernstein} function for all $\gamma\in(0,1)$ and Bernstein functions are exactly the functions that \emph{operate} on the space of continuous negative definite functions in $\mathbb R^d$.
\end{remark}
Finally, we use the following also classical result about characteristic functions (in the probabilistic sense) of stable distributions. The proof can be found for example in \cite{Ch}.

\begin{theorem}\label{t.char} Let $0<\beta\leq 2$. Then the function $\psi_\beta(t)=e^{-|t|^\beta}$, $t\in \mathbb R$, is the Fourier transform of a probability measure on the real line.
\end{theorem}

We are now ready to give the proof of Lemma \ref{l.core}.
\begin{proof}[Proof of Lemma \ref{l.core}] For convenience we recall the definition of the metric $\rho$:
\begin{eqnarray*}
\rho(x)=\sum_{0\leq l\leq n-1} \bigg(\sum_{2^{l-1}<j\leq 2^l} |x_j|^\frac{2^l}{j}\bigg)^\frac{1}{2^l}+\bigg(\sum_{2^{n-1}<j\leq d} |x_j|^\frac{2^n}{j}\bigg)^\frac{1}{2^n}.
\end{eqnarray*}
Now we fix some $0<l\leq n$ and consider the function
\begin{eqnarray*}
\psi^{(l)}(\xi_{2^{l-1}+1},\ldots,\xi_{2^l})= \sum_{2^{l-1}<j\leq 2^l} |\xi_j|^\frac{2^l}{j}.
\end{eqnarray*}
If $l=n$ then it is to be understood that the sum extends up to $d$ but this is of no importance. Now by Theorem \ref{t.char}, for every $2^{l-1}<j\leq 2^l$ there is a probability measure $\mu^{(j,l)}$ such that
\begin{eqnarray*}
\widehat{ \mu^{(j,l)} }(\xi_j)=e^{-|\xi_j|^\frac{2^l}{j}}, \ \  2^{l-1}<j\leq 2^l.
\end{eqnarray*}
This is possible since $1\leq \frac{2^l}{j} <2$. Now we write
\begin{eqnarray*}
\mu^{(l)}\coloneqq \mu^{(2^{l-1}+1,l)}\otimes \cdots \otimes \mu^{(2^l,l)}.
\end{eqnarray*}
By the construction it is obvious that $\widehat{\mu^{(l)}}(\xi_{2^{l-1}+1},\ldots,\xi_{2^l})=e^{-\psi^{(l)}(\xi_{2^{l-1}+1},\ldots,\xi_{2^l})}$. By Bochner's theorem we conclude that the function $e^{-\psi^{(l)}}$ is positive definite (and continuous at $0$ but this is obvious).  We also claim that the function $e^{-t\psi^{(l)}}$ is postive definite for all $t>0$. Indeed, if $\epsilon=t^\frac{1}{2^l}$ then
\begin{eqnarray*}
\widehat {\mu^{(l)}_\epsilon}(\xi_{2^{l-1}+1},\ldots,\xi_{2^l})={\widehat{ \mu^{(l)}}}(\delta_\epsilon(\xi_{2^{l-1}+1},\ldots,\xi_{2^l}))=e^{-\epsilon ^{2^l}\psi^{(l)}(\xi_{2^{l-1}+1},\ldots,\xi_{2^l})}=e^{-t\psi^{(l)}(\xi_{2^{l-1}+1},\ldots,\xi_{2^l})},
\end{eqnarray*}
so by appealing to Bochner's theorem again we get the claim. Theorem \ref{t.schoen} now says that the function $\psi^{(l)}$ is negative definite. However, Lemma \ref{l.power} says that the function
\begin{eqnarray*}
\phi^{(l)}(\xi_{2^{l-1}+1},\ldots,\xi_{2^l})= \big(\sum_{2^{l-1}<j\leq 2^l} |\xi_j|^\frac{2^l}{j}\big)^\frac{1}{2^l}
\end{eqnarray*}
is also negative definite (since $\frac{1}{2^l}\leq 1$.) Thus the function $e^{-\phi^{(l)}}$ is positive definite and obviously continuous with value equal to $1$ at zero. We conclude by Bochner's theorem that there exists a probability measure $\nu^{(l)}$ such that $\widehat {\nu ^{(l)}}(\xi_{2^{l-1}+1},\ldots,\xi_{2^l})=e^{-\phi^{(l)}(\xi_{2^{l-1}+1},\ldots,\xi_{2^l})}$. To finish the proof, we define $\nu=\nu^{(1)}\otimes\cdots\otimes\nu^{(n)}$ and check that $\hat\nu(\xi)=e^{-\rho(\xi)}$ for $\xi\in\mathbb R ^d$.
\end{proof}
Several remarks are in order. Firstly one can actually conclude a bit more in Lemma \ref{l.core}. In particular it is not hard to see there is a non negative $L^1$-function $P^\rho$ such that $dP^\rho(x)=P^\rho(x)dx$, that is that the measure $dP^\rho$ has a non negative density. Indeed the function $\widehat {P^\rho}(\xi)=e^{-\rho(\xi)}$ is in $L^1(\mathbb R^d)$. There are several ways to see that. For example we can use Lemma \ref{l.polar} from which we can actually deduce that
\begin{eqnarray*}
\int_{\mathbb R^d}e^{-\rho(\xi)} d\xi = \alpha|\{x\in\mathbb R^d:\rho(x)\leq 1\}|\int_0 ^\infty e^{-r}r^{\alpha-1}dr<\infty.
\end{eqnarray*}
As a result the function
\begin{eqnarray*}
\phi(x)=\int_{\mathbb R^d}e^{-\rho(\xi)}e^{2\pi i \xi\cdot x} d\xi
\end{eqnarray*}
is a well defined function in $C_o(\mathbb R^d)$. On the other hand we have that $\hat \phi (\xi)=e^{-\rho(\xi)}$ in the sense of distributions so the measure $dP^\rho$ must coincide everywhere with a non negative $L^1$-function. We will call $P^\rho$ the \emph{Poisson kernel associated with the parabolic norm $\rho$}.

Another point is that the connection with negative definite functions should have been expected once we set out to find a convolution semi-group. Indeed, convolution semi-groups are in a one to one correspondance with negative definite functions, the correspondance being the one implied by the discussion in this paragraph. Whenever we have a convolution semi-group $T^t(f)=\mu_t*f$, there exists a (uniquely determined) continuous negative definite function $\psi$ such that $\widehat{\mu_t}(\xi)=e^{-t\psi(\xi)} $. Of course the converse is also true. This theme is fully developed in \cite{J}.

The discussion above leads naturally to the following theorem which is just an application of Stein's general
maximal theorem for symmetric diffusion semi-groups. For more details see \cite{S1}. Recall that for a function
$K\in L^1(\mathbb R^d)$ and a parameter $s>0$ we denote by $K_s$ the parabolic dilation of $K$ as defined in
page \pageref{p.pardil}.
\begin{theorem}\label{t.semigroup} Let $\rho$ be the parabolic norm defined in \eqref{e.metric} and let $P^\rho$ its associated Poisson kernel. For $t>0$ we define the family of operators $T^t:L^p(\mathbb R^d)\rightarrow L^p(\mathbb R^d)$, $1\leq p \leq \infty$ as
\begin{eqnarray*}
T^t(f)(x)=f*P^\rho _t(x), f\in L^p(\mathbb R ^d).
\end{eqnarray*}
Then the family $\{T^ t\}_{0<t<\infty}$ is a semi-group of operators since by construction $T^{t_1}\circ T^{t_2}=T^{t_1+t_2}$ for every $t_1,t_2>0$ and $T^0=Id$. We also have that $\lim_{t\rightarrow 0}T^tf =f$ in $L^2(\mathbb R ^d)$. The family $\{T^t\}$ also satisfies the following properties:
\begin{enumerate}
\item[(i)]{$\norm T^t f.L^p(\mathbb R^d).\leq \norm f.L^p(\mathbb R^d).$, $t>0$, $1\leq p \leq \infty$ (contraction property).}
\item[(ii)]{For every $t>0$, $T^t$ is a self adjoint operator in $L^2(\mathbb R ^d)$ (symmetry property).}
\item[(iii)]{$T^t f \geq 0$ if $f\geq 0$, $t>0$ (positivity property).}
\item[(iv)]{$T^t1=1$, $t>0$ (conservation property).}
\end{enumerate}
Thus the family $\{T^t\}$ is a symmetric diffusion semi-group. Let $T^*(f)(x)=\sup_{t>0}T^t(f)(x)$. Then
\begin{eqnarray*}
\norm T^*(f).L^p(\mathbb R^d). \leq c_p \norm f.L^p(\mathbb R ^d). ,\ \  1<p\leq \infty \ , f\in L^p(\mathbb R^d),
\end{eqnarray*}
where $c_p$ depends only on $p$.
\end{theorem}

Theorem \ref{t.semigroup} will be applied as follows. Let $d\sigma$ be the measure defined in \eqref{e.sigma} and $P^\rho$ is the Poisson kernel associated with the parabolic norm $\rho$. We can now estimate our maximal function like in \eqref{e.notoy} where $P^\rho(x)dx$ plays the role of the measure $d\eta$ in \eqref{e.notoy}:
\begin{eqnarray*}
\mathcal M ^d _\sigma(x) \leq \sup_{s>0}(f*P^\rho _s)(x)+\big(\sum_{k\in \mathbb Z} |f*(d\sigma-dP^\rho)_{2^k}(x)|^2\big)^\frac{1}{2}.
\end{eqnarray*}
By theorem \ref{t.semigroup} the first term in the above estimate is bounded in $L^2$ (in fact in all $L^p$, $1<p\leq\infty$) with bounds that do not depend on the dimension $d$. We write $d\nu=d\sigma-dP^\rho$ and we define the square function
\begin{eqnarray}\label{e.square}
S(f)(x)=\big(\sum_{k\in \mathbb Z} |(f*d\nu_{2^k})(x)|^2\big)^\frac{1}{2}.
\end{eqnarray}
Theorem \ref{t.main} now reduces to the following:

\begin{theorem}\label{t.square} There exists an absolute constant $c>0$ such that
\begin{eqnarray*}\norm S(f).L^2(\mathbb R ^d).\leq c \log d \ \norm f.L^2(\mathbb R ^d)..
\end{eqnarray*}
\end{theorem}
The proof of this theorem is the content of the following section.

\section{The square function estimate.}\label{s.square}
This section is devoted to proving the inequality
\begin{eqnarray*}
\norm S(f).L^2(\mathbb R^d).\leq c\log d \ \norm f.L^2(\mathbb R^d). ,
\end{eqnarray*}
where $c>0$ is an absolute constant. Remember that $S(f)$ is the square function defined in \eqref{e.square} with respect to the measure $d\nu=d\sigma-dP^\rho $. Using Plancherel's theorem we get:
\begin{eqnarray*}
\norm \mathcal S(f).L^2(\mathbb R^d).&=& \NOrm \bigg(\sum_{k\in \mathbb Z}|f*d\nu_{2^k}|^2\bigg)^\frac{1}{2}.L^2(\mathbb R^d). =\bigg( \sum_{k\in \mathbb Z}\int_{\mathbb R ^d}|(f*d\nu_{2^k})(x)|^2 dx\bigg) ^\frac{1}{2}\\
&=&  \bigg( \sum_{k\in \mathbb Z}\int_{\mathbb R ^d}|\hat f(\xi)|^2|\widehat{\nu_{2^k}}(\xi)|^2 d\xi\bigg) ^\frac{1}{2}\\ &\leq& \sup_{\xi\in\mathbb R ^d} \norm \widehat {\nu_{2^k}}(\xi).\ell^2(\mathbb Z). \norm f .L^2(\mathbb R^d). .
\end{eqnarray*}
Here of course we denote $\norm \widehat {\nu_{2^k}}(\xi).\ell^2(\mathbb Z).=\big(\sum_{k\in \mathbb Z}|\nu_{2^k}(\xi)|^2\big)^\frac{1}{2}$. So in order to prove the theorem it is enough to prove that
\begin{eqnarray*}
\sup_{\xi\in\mathbb R ^d} \norm \widehat {\nu_{2^k}}(\xi).\ell^2(\mathbb Z). \leq c \log d, \ \ d>1.
\end{eqnarray*}

Estimating $|\widehat {\nu_{2^k}}(\xi)|$ for large $\xi$ amounts to estimating suitable oscillatory integrals. We will digress a bit now to state a lemma that is the appropriate estimate for the proof.

\subsection{Oscillatory integral estimates}
We will need to estimate an oscillatory integral of the form
\begin{eqnarray*}
\int_a ^b e^{ip(t)}dt,
\end{eqnarray*}
where $P(t)=b_1t+b_2t^2+\cdots+b_dt^d$ is a real polynomial of degree (at most) $d$ with zero constant term. As is well known oscillatory integral estimates are in a sense equivalent to sublevel set estimates. We have the following lemma due to Vinogradov \cite{V}:

\begin{lemma}\label{l.vinogradov} Let $p(t)=b_0+b_1 t+\cdots+b_d t^d$ be a real polynomial of degree d. Then,
\begin{eqnarray*}
|\{t\in[a,b]:|p(t)|\leq \delta\}| \lesssim \max(|a|,|b|)\bigg(\frac{\delta}{\max_{0\leq k \leq d} |b_k|}\bigg)^\frac{1}{d}.
\end{eqnarray*}
\end{lemma}

Now the estimate for the corresponding oscillatory integral is trivial. Indeed, let $\delta >0$ and write
\begin{eqnarray*}
\ABs{\int_a ^b e^{ip(t)}dt}&\leq& |\{t\in[a,b]:|p^\prime(t)|<\delta\}| +\ABs{\int_{\{t\in[a,b]:|p^\prime(t)|\geq\delta\}} e^{ip(t)}dt}\lesssim
\\ &\lesssim&\max\{|a|,|b|\}\bigg(\frac{\delta}{\max_{1\leq k \leq d} |b_k|}\bigg)^\frac{1}{d-1}+\frac{d}{\delta}.
\end{eqnarray*}
The first term comes from Lemma \ref{l.vinogradov}. For the second term note that $p^\prime$ changes
monotonicity at most $d-1$ times in $[a,b]$. Thus we can split the set $\{t\in[a,b]:|p^\prime(t)|\geq\delta\}$
into $\mathcal O(d)$ intervals where $p^\prime$ is monotonic and integrate by parts. Optimizing in $\delta$
gives the following lemma:

\begin{lemma}\label{l.oscillatory} Let $p(t)=b_1 t+\cdots+b_d t^d$ where $b_j\in\mathbb R$ for $j=1,\ldots,d$. Then for all $a,b\in \mathbb R$ with $a<b$ we have
\begin{eqnarray*}
\ABs{\int_a ^b e^{ip(t)}dt} \lesssim \frac{\max(|a|,|b|)^{1-\frac{1}{d}}}{(\max_{1\leq k \leq d}
|b_k|)^\frac{1}{d}}.
\end{eqnarray*}
\end{lemma}

We now have all the ingredients to complete the proof of Theorem \ref{t.square} and thus that of Theorem \ref{t.main}.
\begin{proof}[Proof of Theorem \ref{t.square}] Remember that the square function estimate of the theorem will be established if we show:
\begin{eqnarray}\label{e.nuhats}
\sup_{\xi\in\mathbb R ^d} \norm \widehat {\nu_{2^k}}(\xi).\ell^2(\mathbb Z). \leq c \log d, \ \ d>1.
\end{eqnarray}
A moment's reflection will now allow us to assume that $d=2^n$ for some $n\in \mathbb N$, thus simplifying a bit the argument that follows. Indeed, suppose that we have proved \eqref{e.nuhats} with $2^n$ in the place of $d$:
\begin{eqnarray}\label{e.induction}
\sup_{\xi\in\mathbb R ^{2^n}} \norm \widehat{ {\tilde \nu }_{2^k}}(\xi).\ell^2(\mathbb Z). \leq c \log 2^n,
\end{eqnarray}
where $\tilde\nu$ is defined in the obvious way on $\mathbb R^{2^n}$. For general $d$, suppose that $2^{n-1}<d\leq 2^n$ and for $\xi\in\mathbb R^d$, consider $\bar\xi\in \mathbb R^{2^n}$ where $\bar \xi_j= \xi_j$ for $1\leq j \leq d$ and $\bar \xi_j= 0$ for $d+1\leq j \leq 2^n$. Now clearly,
\begin{eqnarray*}
\sup_{\xi\in\mathbb R ^d} \norm \widehat { \nu _{2^k}}(\xi).\ell^2(\mathbb Z). = \sup_{\xi\in\mathbb R ^d} \norm \widehat{\tilde \nu _{2^k}}(\bar\xi).\ell^2(\mathbb Z). \leq \sup_{\xi\in\mathbb R ^{2^n}} \norm \widehat{ {\tilde {\nu} }_{2^k}}(\xi).\ell^2(\mathbb Z). \leq c \log 2^n \lesssim \log d,
\end{eqnarray*}
which proves the claim. We can and will therefore assume that $d=2^n$.

For $n=0,1,2,\ldots,$ we define
\begin{eqnarray*}
c_n\coloneqq \sup_{\xi\in\mathbb R^{2^n}}\norm \widehat {\nu _{2^k}}(\xi).\ell^2(\mathbb Z). .
\end{eqnarray*}
We will prove that $c_n\leq c_{n-1}+c$ for all $n\geq 1$, using induction on $n$. Clearly this proves \eqref{e.induction} and thus \eqref{e.nuhats}.

For $n=0$ and $\xi_1\in\mathbb R$ we have
\begin{eqnarray*}
\widehat{\nu_{2^k}}(\xi_1)=\int_{\frac{1}{2}\leq|t|<1}e^{-2\pi i \xi_1 2^k t}dt - e^{-2^k \rho(\xi_1)}
\end{eqnarray*}
Now on the one hand
\begin{eqnarray*}
\ABs{\int_{\frac{1}{2}\leq|t|<1}e^{-2\pi i \xi_1 2^k t}dt }\lesssim \frac{1}{2^k|\xi_1|},
\end{eqnarray*}
and
\begin{eqnarray*}
e^{-2^k \rho (\xi_1)}=e^{-2^k|\xi_1|}\leq \frac{1}{2^k |\xi_1|}.
\end{eqnarray*}
On the other hand
\begin{eqnarray*}
\ABs{\int_{\frac{1}{2}\leq|t|<1}e^{-2\pi i \xi_1 2^k t}dt - e^{-2^k \rho(\xi_1)}}&\leq&
\ABs{\int_{\frac{1}{2}\leq|t|<1}e^{-2\pi i \xi_1 2^k t}dt-1}+\Abs{ e^{-2^k \rho(\xi_1)}-1}\\
&\lesssim&  2^k |\xi_1|.
\end{eqnarray*}
Since we have the estimate
\begin{eqnarray*}
\widehat{\nu_{2^k}}(\xi_1)\lesssim \min \bigg(2^k|\xi_1|,\frac{1}{2^k|\xi_1|}\bigg),
\end{eqnarray*}
it is now trivial to estimate the $\ell^2$ norm as follows:
\begin{eqnarray*}
\norm \widehat {\nu _{2^k}}(\xi_1).\ell^2(\mathbb Z). \leq \norm \widehat {\nu _{2^k}}(\xi_1).\ell^2(2^k>|\xi_1|^{-1}).+\norm \widehat {\nu _{2^k}}(\xi_1).\ell^2(2^k\leq|\xi_1|^{-1}).\lesssim 1,
\end{eqnarray*}
for all $\xi_1\in\mathbb R$. So the first step of the induction works.

We now take $n\geq 1$ and $\xi\in\mathbb R ^{2^n}$ and we define $y\in\mathbb R^{2^{n-1}}$ as $y=(\xi_1,\ldots,\xi_{2^{n-1}})$. We also define $2^{n-1}<j_o\leq 2^n$ such that $|\xi_{j_o}|^\frac{1}{j_o}= \max_{2^{n-1}<j\leq 2^n}|\xi_l|^\frac{1}{l}\eqqcolon A^{-1}$. We have
\begin{eqnarray*}
\norm \widehat {\nu _{2^k}}(\xi).\ell^2(\mathbb Z). &\leq& \norm \widehat {\nu _{2^k}}(\xi).\ell^2(2^k>A).+\norm \widehat {\nu _{2^k}}(\xi)-\widehat {\nu _{2^k}}(y).\ell^2(2^k\leq A).+ \norm \widehat {\nu _{2^k}}(y).\ell^2(2^k\leq A).\\ &\leq& \norm \widehat {\nu _{2^k}}(\xi).\ell^2(2^k>A).+\norm \widehat {\nu _{2^k}}(\xi)-\widehat{ \nu _{2^k}}(y).\ell^2(2^k\leq A).+c_{n-1}\\
&\eqqcolon&I+II+c_{n-1}.
\end{eqnarray*}
We estimate $I$ by $I\leq \norm  \widehat{\sigma_{2^k}}(\xi).\ell^2({2^k>A}). +\Norm \widehat{ P^{\rho} _{2^k}} (\xi) .\ell^2({2^k>A}).$. Using Lemma \ref{l.oscillatory} and writing $2^n=d$ we have
\begin{eqnarray*}
|\widehat{\sigma_{2^k}}(\xi)|=\ABs{\int_{\frac{1}{2}\leq|t|<1}e^{-2\pi i (\xi_1 2^k t+\cdots+\xi_d 2^{kd}t^d)}dt }\lesssim \frac{1}{\big(\max_{1\leq j \leq d}|\xi_j2^{kj}|\big)^\frac{1}{d}}\leq \frac{1}{|\xi_{j_o}|^\frac{1}{d}2^\frac{kj_o}{d}}.
\end{eqnarray*}

For the ``Poisson semigroup'' $\widehat {P^\rho}$ we have
\begin{eqnarray*}
\widehat {P^{\rho} _{2^k}}(\xi)=e^{-2^k \rho(\xi)}\leq \frac{1}{2^k \rho(\xi)}\leq
\frac{1}{2^k|\xi_{j_o}|^\frac{1}{j_o}}.
\end{eqnarray*}
Summing up these estimates we get
\begin{eqnarray*}
I=\norm \widehat {\nu _{2^k}}(\xi).\ell^2(2^k>A).\lesssim \Bigg(\sum_{2^k>A}
\frac{1}{|\xi_{j_o}|^\frac{2}{d}2^\frac{2kj_o}{d}}\Bigg)^\frac{1}{2}+
\Bigg(\sum_{2^k>A}\frac{1}{2^{2k}|\xi_{j_o}|^\frac{2}{j_o}}\Bigg)^\frac{1}{2}\lesssim
\frac{1}{1-2^{-\frac{j_o}{d}}}\lesssim \frac{d}{j_o} \lesssim 1.
\end{eqnarray*}

We continue with the estimate for $II$. Here we write
\begin{eqnarray*}
\norm \widehat {\nu _{2^k}}(\xi)-\widehat {\nu _{2^k}}(y).\ell^2(2^k\leq A).\leq \norm \widehat {\sigma _{2^k}}(\xi)-\widehat {\sigma _{2^k}}(y).\ell^2(2^k\leq A).+\Norm \widehat {P^{\rho} _{2^k}}(\xi)-\widehat {P^{\rho} _{2^k}}(y).\ell^2(2^k\leq A). .
\end{eqnarray*}
We estimate the two summands separately. We have
\begin{eqnarray*}
\Abs{\widehat {\sigma _{2^k}}(\xi)-\widehat {\sigma _{2^k}}(y)}&\leq& \int_{\frac{1}{2}\leq |t|<1}\ABs {e^{-2\pi i (2^k\xi_1t+\cdots+\xi_d 2^{kd}t^d)}-e^{-2\pi i (2^k\xi_1t+\cdots+\xi_\frac{d}{2}2^{k\frac{d}{2}} t^\frac{d}{2})}} dt\\
&\lesssim& \sum_{j=2^{n-1}} ^{2^n} \frac{2^{kj}|\xi_j|}{j+1} \lesssim  2^{kj_1}|\xi_{j_1}|,
\end{eqnarray*}
for some $2^{n-1}<j_1\leq 2^n$. Summing in $k$ in $\{2^k \leq A\}$ we thus get
\begin{eqnarray*}
\Norm \widehat {\sigma _{2^k}}(\xi)-\widehat{ \sigma _{2^k}}(y).\ell^2(2^k\leq A).&\leq& |\xi_{j_1}|\ \bigg( \sum_{2^k \leq A} 2^{2kj_1}\bigg)^\frac{1}{2} \lesssim\frac{1}{1-2^{-j_1}} A^{j_1} |\xi_{j_1}| \lesssim 1.
\end{eqnarray*}

Finally using the definition of the metric $\rho$ we can write
\begin{eqnarray*}
\Abs{\widehat {P^{\rho} _{2^k}}(\xi)-\widehat {P^{\rho} _{2^k}}(y) }  &=&\Abs{e^{-2^k\rho(\xi)}-e^{-2^k\rho(y)}}\leq 2^k |\rho(\xi)-\rho(y)|\\
&=& 2^k \bigg( \sum_{2^{l-1}<j\leq 2^l} |\xi_j|^\frac{2^l}{j}  \bigg)^\frac{1}{2^l}\leq 2^k |\xi_{j_o}|^\frac{1}{j_o} (2^l)^\frac{1}{2^l}\lesssim 2^k |\xi_{j_o}|^\frac{1}{j_o}.
\end{eqnarray*}
As a result
\begin{eqnarray*}
\Norm \widehat {P^{\rho} _{2^k}}(\xi)-\widehat {P^{\rho} _{2^k}}(y).\ell^2(2^k\leq A). \lesssim \bigg( \sum_{2^k\leq A}  2^{2k} |\xi_{j_o}|^\frac{2}{j_o}\bigg)^\frac{1}{2}\lesssim 1.
\end{eqnarray*}
Thus $II\lesssim 1$ and so $c_n\leq c_{n-1}+c$ for all $n\geq 1$ which completes the proof.
\end{proof}

 \end{document}